\documentclass{amsart}
\usepackage{amsfonts, amssymb}

\def\A{\mathcal A}
\def\B{\mathcal B}
\def\e{\varepsilon}

\newtheorem{Th}{Theorem}
\newtheorem{Cor}{Corollary}

\newtheorem{lemma}{Lemma}

\newtheorem{defn}{Definition}
\newtheorem{prop}{Proposition}
\newcommand{\Z}{\ensuremath{\mathbb Z}}
\newcommand{\N}{\ensuremath{\mathbb N}}
\newcommand{\R}{\ensuremath{\mathbb R}}
\newcommand{\U}{\ensuremath{\mathcal U}}

\begin{document}
\date{\today}
\title{Dense sets of integers with prescribed representation functions}

\author{Javier Cilleruelo}
\thanks{The work of J.C. was supported by Grant MTM 2005-04730 of MYCIT (Spain)}
\address{Departamento de Matem\'aticas\\
Universidad Aut\'onoma  de Madrid\\
28049 Madrid, Spain }
\email{franciscojavier.cilleruelo@uam.es}

\author{Melvyn B. Nathanson}
\thanks{The work of M.B.N. was supported
in part by grants from the NSA Mathematical Sciences Program and
the PSC-CUNY Research Award Program}
\address{Department of Mathematics\\
Lehman College (CUNY)\\
Bronx, New York 10468}
\email{melvyn.nathanson@lehman.cuny.edu}

\begin{abstract}
Let $\A$ be a set of integers and let $h\ge 2$.   For every integer $n$, let $r_{\A, h}(n)$ denote the number of   representations of $n$ in the form $n=a_1+\cdots +a_h,$ where $a_i\in    \A$ for $1\le i \le h$, and $a_1\le \cdots \le a_h$. The function $r_{\A,h}:\Z \to \mathbf N$,
where $\mathbf N=\N \cup \{0,\infty\}$, is  the {\it representation function of order $h$  for $\A$}.

We prove that every function $f:\Z \rightarrow \mathbf N$
satisfying $\liminf_{|n|\to \infty}f(n)\ge g$  is the
representation function of order $h$ for a sequence $\A=\{ a_k\}$
of integers, and that $\A$ can be constructed so that it increases
``almost" as slowly as any given $B_h[g]$ sequence. In particular,
for every $\varepsilon >0$ and $g\ge g(h,\varepsilon),$ we can
construct a sequence $\A$  satisfying $r_{\A,h}=f$ and $\A(x)\gg
x^{(1/h)-\varepsilon}$.
\end{abstract}

     \maketitle

 \section{Introduction}

Let $\A$ be a set of integers and let $h\ge 2$.
For every integer $n$, let $r_{\A, h}(n)$ denote the number of
     representations of $n$ in the form $$n=a_1+\cdots +a_h$$
     where
     $$
     a_1\le \cdots \le a_h \quad\text{and}\quad
     a_i\in
     \A \quad\text{for}\quad 1\le i \le h.$$ The function $r_{\A,h}:\Z \to \mathbf N$ is the {\it representation function of order $h$
     for $\A$}, where $\mathbf N=\N \cup \{0,\infty\}$.

Nathanson  proved \cite{Nath}  that  any function $f:\Z \to
\mathbf N$ satisfying $ \liminf_{|n|\to \infty}f(n)\ge 1 $ is the
{\em representation function of order} $h$ of a set of integers
$\A$
 such that\footnote{The notation $f(x)\gg g(x)$
 means that there exists a
constant $C>0$ such that $f(x)\ge Cg(x)$ for $x$ large enough}
\begin{equation}\label{A}
\A(x)\gg x^{1/(2h-1)},\end{equation}
where $\A(x)$
counts the number of elements  $a \in \A$ with $|a| \leq x$.  It is an open problem to determine how dense the sets $\A$
can be.

\smallskip

In this paper we study the connection between this problem and the
problem of finding dense $B_h[g]$ sequences.   We recall that a set
$\B $ of nonnegative integers  is called a $B_h[g]$ sequence if
$$r_{\B,h}(n)\le g$$ for every nonnegative integer $n$.  It is usual
to write $B_h$ to denote  $B_h[1]$ sequences.

\smallskip

Luczak and Schoen proved that any $B_h$ sequence satisfying an
additional kind of Sidon property (see \cite{L-S} for the
definition of this property, which they call the $S_h$ property) can
be enlarged to obtain a sequence with any prescribed
representation function given $f$ satisfying that $\liminf_{|x|\to
\infty}f(x)\ge 1$.  In particular, since they prove that there
exists a $B_h$ sequence $\A$ satisfying the $S_h$ property with
$\A(x)\gg x^{1/(2h-1)}$, they recover Nathanson's result.

\smallskip

In this paper we prove that any $B_h[g]$ sequence can be modified
slightly to have any prescribed representation function $f$
satisfying $\liminf_{|x|\to \infty}f(x)\ge g$. Our main
theorem is the following.

\begin{Th}\label{thm}
Let $f:\Z\rightarrow \mathbf N $ any function such that
$\liminf_{|n|\to \infty}f(n)\ge g$  and let $\B$ be any $B_h[g]$
sequence. Then, for any decreasing function $\epsilon(x)\to 0$ as
$x\to \infty $, there exists a sequence $\A$ of integers such that
\[
r_{\A,h}(n)=f(n)\quad   \text{ for all $  n\in \Z$}
\]
and
\[
\A(x)\gg \B(x\epsilon(x)).
\]
\end{Th}

It is difficult problem to construct dense $B_h$ sequences. A trivial
counting argument gives $$\B(x)\ll x^{1/h}$$ for these sequences.
On the other hand, the greedy algorithm shows that there exists a
$B_h$ sequence $\B$ such that \begin{equation}\label{Bx}\B(x)\gg
x^{1/(2h-1)}.\end{equation}

\smallskip

For  $B_2$ sequences, also called Sidon sets, Ruzsa proved
\cite{Ruzsa} that there exists a Sidon set $\B $ such that
\begin{equation}\label{Ruzsa} \B(x)\gg x^{\sqrt
2-1+o(1)}\end{equation}

This result and Theorem~\ref{thm} give the following corollary.

\begin{Cor}
Let $f:\Z\rightarrow \mathbf N$ any function such that
$\liminf_{|n|\to \infty}f(n)\ge 1$. Then there exists a sequence
of integers $\A$ such that
\[ r_{\A,2}(n)=f(n)\quad \text{ for all $n\in
\Z$}
\]
and
\[
\A(x)\gg x^{\sqrt 2-1+o(1)}.
\]
\end{Cor}

This result gives an affirmative answer to the third open problem
in \cite{Chen}, which was also posed previously in  \cite{Nath2}.
Unfortunately, nothing better than (\ref{Bx}) is known for $B_h$
sequences for $h\ge 3$.   Erd\H os and Renyi \cite{Erdos} proved
that, for any $\epsilon>0$, there exists a positive integer $g$
and a $B_2[g]$ sequence $\B$ such that $ \B(x)\gg
x^{(1/2)-\epsilon}.$  They claimed that the same method could be
extended  to $B_{h}[g]$ sequences, but a serious problem with
non-independent events appears when $h\ge
 3$.

 \smallskip

Vu~\cite{Vu}, using  ideas from Erd\H os and Tetali \cite{Er-Te} to study
a related problem for bases of order $h$, proved that for any $\epsilon>0$, there exist an integer $g$ and a $B_h[g]$
sequence  $\B$ such that
\[
\B(x)\gg
x^{1/h-\epsilon}.
\]
This result and Theorem \ref{thm} imply the next corollary

\begin{Cor}
Given $h\ge 2$, for any $\varepsilon >0$, there exists
$g=g(h,\varepsilon)$ such that, for any function $f:\Z\rightarrow
\mathbf N $ satisfying $\liminf_{|n|\to \infty}f(n)\ge g$,
there exists a sequence $\A$ of integers such that
\[
r_{\A,h}(n)=f(n)\quad \text{ for all $n\in \Z$}
\]
and
\[
\A(x)\gg x^{\frac 1h-\varepsilon}.
\]
\end{Cor}

The construction in \cite{Nath} for the set $\A$ satisfying the
growth condition~(\ref{A}) was based on the greedy algorithm. In
this paper we construct the set $\A$ by adjoining a very sparse
sequence $\U=\{ u_k\}$ to a suitable $B_h[g]$ sequence $\B$.  This
idea was used in \cite{Ci-Nath}, but  in a simpler way, to
construct dense {\em perfect difference sets}, which are sets such
that every  nonzero integer has a unique representation  as a
difference of two elements of $\A$.  The proof of the main theorem
in \cite{Ci-Nath} can be adapted easily to our problem in the
simplest case $h=2$.

\begin{Th}\label{thm2}
Let $f:\Z\rightarrow \mathbf N$ be a function such that
$\liminf_{|n|\to \infty}f(n)\ge g$,  and let $\B$ be a $B_2[g]$
sequence. Then there exists a sequence of integers $\A$ such that
\[
r_{\A,2}(n)=f(n)\quad \text{ for all $n\in \Z$}
\]
and
\[
\A(x)\gg \B(x/3).
\]
\end{Th}

We omit the proof  because it is very close to the proof of the
main theorem in \cite{Ci-Nath}. Unfortunately, that proof cannot
be adapted to the case $h\ge 3$. We need another definition  of a
``suitable" $B_h$ set.  In section \S 2 we shall show how to
modify a $B_h[g]$ sequence $\B$ so that it becomes ``suitable." We
do this by applying the ``Inserting Zeros Transformation" to an
arbitrary $B_h[g]$ set. This is the main ingredient in the proof
of Theorem~ \ref{thm}.

\smallskip

Chen \cite{Chen} has proved that for any $\epsilon >0$ there
exists a unique representation basis $\A$ (that is, a set $\A$
with $r_{\A,2}(k)=1$ for all integers $k\ne 0$) such that
$\limsup_{x\to \infty} \A(x)/x^{1/2-\epsilon}>1$. J. Lee
\cite{Lee} has improved this result by proving that for any
increasing function $\omega$ tending to infinity there exists a
unique representation basis $\A$ such that $\limsup_{x\to
\infty}\A(x)\omega(x)/\sqrt x>0$.

\smallskip

Theorem \ref{thm2} and  the classical constructions of Erd\H
os~\cite{stoh55} and Kruck\" eberg~\cite{kruc61}  of infinite
Sidon sets $\B$ such that $\limsup_{x\to \infty}\B(x)/\sqrt x>0$
provide a unique representation basis $\A$ such that
$\limsup_{x\to \infty}\A(x)/\sqrt x>0$.  Indeed, we can easily
adapt the proof of Theorem 1.3 in \cite{Ci-Nath} to the case of
the additive representation function $r(n)$ (instead of the
subtractive representation function $d(n)=\# \{ n=a-a',\ a,a'\in
\A\}$).
\begin{Th}
There exists a unique representation basis $\A$ such
that
$$\limsup_{x\to \infty}\frac{\A(x)}{\sqrt x}\ge \frac 1{\sqrt 2}.$$
\end{Th}
Again we omit the proof because it is very close to the proof of
Theorem 1.3 in \cite{Ci-Nath}.

\smallskip

Theorem above answers affirmatively the first open problem in
\cite{Chen}. Note also that if $\A$ is an infinite Sidon set of
integers, then the set $$\A'=\{ 4a: a\ge 0\}\cup \{-4a+1 : a<0\}$$
is also a Sidon set and, in this case, $\liminf|\A\cap
(-x,x)|/\sqrt x=\liminf \A'(4x)/\sqrt x$. A well known result of
Erd\H os states that $\liminf \B(x)/\sqrt x=0$ for any Sidon set
$\B$. Then the above limit is zero, so it answers negatively the
second open problem in \cite{Chen}.

\smallskip

 We do not  know if it is possible to obtain a similar result for $h \geq 3,$
because it is open problem to determine if there exists an infinite $B_h$
sequence $\B$ with $\limsup_{x\to \infty }\B(x)/x^{1/h}>0$. It is
easy, however, to prove that for any function $\omega$ tending to
infinity there exists a unique representation basis of order $h$
such that $\limsup_{x\to \infty}\B(x)\omega(x)/x^{1/h}>1$. We can
construct the set $\B $ as follows: Let $x_1,\dots , x_k,\dots $ be a
sequence of positive integers such that
$\omega(x_k)>(hx_{k-1})^{1/h}$ and consider, for each $k$,  a
$B_h$ sequence $\B_k\subset [1,x_k/(hx_{k-1})]$ with $|\B_k|\gg
(x_k/(hx_{k-1}))^{1/h}$.  The set
$\B=\cup_k (hx_{k-1})\ast\B_k$ satisfies the conditions, where we
use the notation $t\ast \A=\{ ta,\ a\in \A\}$.

The construction above and Theorem \ref{thm} yield the following
Corollary, which extends Theorem 6 in \cite{Lee} in several ways.

\begin{Cor}
Let $f:\Z\rightarrow \mathbf N $ any function such that
$\liminf_{|n|\to \infty}f(n)\ge 1$.  For any increasing function
$\omega$ tending to infinity there exists a set  $\A$ such that
$r_{\A,h}(n)=f(n)$ for all integers $n$, and
$$
\limsup_{x\to \infty}\A(x)\omega(x)/x^{1/h}>0.
$$
\end{Cor}

\section{The Inserting Zeros Transformation}

Consider the binary expansion of the elements of a set $\B$ of
positive integers. We will modify these integers by inserting
strings of zeros at fixed places. We will see that this
transformation of the set $\B$ preserves certain additive properties.

\smallskip

In this paper we denote by $\gamma $ any strictly increasing
function $\gamma :\N_0 \to \N_0$ with $\gamma (0)=0$.   For every
positive integer $r$, we define the ``Inserting Zeros
Transformation'' $T_{\gamma }^r$ by
\begin{equation}\label{se}
T_{\gamma }^r\left(\sum_{i\ge 0}\e_i2^i \right)=\sum_{k\ge
0}2^{2rk}\sum_{i=\gamma (k)}^{\gamma (k+1)-1}\e_i 2^i.
\end{equation}
In other words, if the integer $b$ has the binary expansion
\[
b=\e_0\cdots \e_{\gamma (1)-1}\e_{\gamma (1)}\cdots \e_{\gamma
(2)-1}\e_{\gamma (2)}\cdots \e_{\gamma (k)-1}\e_{\gamma
(k)}\cdots,
\]
then
\[
T_{\gamma }^r(b) =\e_0\cdots \e_{\gamma (1)-1} \underbrace{0\cdots
0}_{2r} \e_{\gamma (1)} \cdots \e_{\gamma
(2)-1}\underbrace{0\cdots 0}_{2r} \e_{\gamma (2)}\cdots \e_{\gamma
(k)-1}\underbrace{0\cdots 0}_{2r}\e_{\gamma (k)}\cdots
\]
Note that if $b < b',$ then $T_{\gamma }^r(b) < T_{\gamma
}^r(b').$ We define the set
\begin{equation}\label{B}
T_{\gamma }^r(\B)=\{ T_{\gamma }^r(b): b\in \B\}.
\end{equation}

The next proposition proves that the function $T_{\gamma }^r$
preserves some Sidon properties.

\begin{prop}\label{B[g]}
Let $2r> \log_2h$.   If $b_1,\ldots,b_h,b'_1,\ldots, b'_h$ are positive integers such that
\[
T_{\gamma }^r(b_1)+\cdots +T_{\gamma }^r(b_h)=T_{\gamma
}^r(b'_1)+\cdots + T_{\gamma }^r(b'_h),
\]
then
\[
b_1+\cdots +b_h=b'_1+\cdots +b'_h.
\]
In particular, if $\B$ is a $B_h[g]$ set and $2r\ge \log_2h$, then
$T_{\gamma }^r(\B)$ is also a $B_h[g]$ set.
\end{prop}

\begin{proof}
We write
\begin{equation}\label{tk}
t_k=\sum_{i=\gamma (k)}^{\gamma (k+1)-1}\e_i(b_1)2^i+\cdots
+\sum_{i=\gamma (k)}^{\gamma (k+1)-1}\e_i(b_h)2^i.
\end{equation}

For any $k\ge 1$ we define the integer
\begin{equation}\label{mk}
m_k=2^{2rk + \gamma (k)}.
\end{equation}

It follows from~(\ref{se}), (\ref{tk}) and  (\ref{mk}) that
\[
T_{\gamma }^r(b_1)+\cdots +T_{\gamma }^r(b_h)\equiv
\sum_{j=0}^{k-1}2^{2rj}t_j \pmod{m_{k}}.
\]
Since $T_{\gamma }^r(b_1)+\cdots +T_{\gamma }^r(b_h)=T_{\gamma
}^r(b'_1)+\cdots + T_{\gamma }^r(b'_h),$ we have
\[
\sum_{j=0}^{k-1}2^{2rj}t_j\equiv
\sum_{j=0}^{k-1}2^{2rj}t'_j\pmod{m_{k}}.
\]

Notice that
\[
0\le \sum_{j=0}^{k-1}2^{2rj}t_j\le
2^{2r(k-1)}\sum_{j=0}^{k-1}t_j\le 2^{2r(k-1)}h\sum_{i=0}^{\gamma
(k)-1}2^i<2^{2r(k-1)}2^{2r}2^{\gamma (k)}=m_{k},
\]
and the same inequality works for $\sum_{j=0}^{k-1}2^{rj}t'_j$.
Then
\[
\sum_{j=0}^{k-1}2^{2rj}t_j=\sum_{j=0}^{k-1}2^{2rj}t'_j.
\]
It follows that $t_k=t'_k$
for all $k \geq 0$, and so
\[
b_1+\cdots +b_h=\sum_{k\ge 0}t_k=\sum_{k\ge 0}t'_k=b'_1+\cdots
+b'_h.
\]
This completes the proof.
\end{proof}

\begin{defn}
For all integers $m\ge 2$ and $x,$ let
\[
 \| x \|_m=\min \{ |y|,\ x \equiv y \pmod m\}.
\]
\end{defn}

Note that $\| x_1+ x_2\|_m \leq \|x_1\|_m + \|x_2\|_m$ for all integers $x_1$ and $x_2.$  Also, if $\|x\|_m\ne \|x'\|_m$ for some $m$, then $x\not \equiv
x'\pmod m$ and so $x\ne x'$.

\begin{prop}\label{p1}
For $k \geq 1$ and  for any positive integer $b$
\[
\|T_{\gamma }^r({b}) \|_{m_k} <  \frac{m_k}{2^{2r}},
\]
where $m_k$ is defined in (\ref{mk}).
\end{prop}

\begin{proof}
Let $b=\e_0 \e_1 \e_2\ldots$ be the binary expansion of $b.$ Then
\[
T_{\gamma }^r({b})\equiv \sum_{j=0}^{k-1}2^{2rj}\sum_{i=\gamma
(j)}^{\gamma (j+1)-1}\e_i 2^i \pmod{m_k}
\]
and
\[
0  \leq \sum_{j=0}^{k-1}2^{2rj}\sum_{i=\gamma (j)}^{\gamma
(j+1)-1}\e_i2^i \le \sum_{l=0}^{2r(k-1)+ \gamma (k)-1}2^l<
\frac{m_k}{2^{2r}}.
\]
This completes the proof.
\end{proof}

\section{Proof of Theorem \ref{thm}}

\subsection{Two auxiliary sequences} Consider the sequence $\{ z_j\}_{j=1}^{\infty}$
defined by
\begin{equation}\label{z}
z_j=j-[\sqrt j]([\sqrt j]+1).
\end{equation}
For every positive integer $j$ there is a unique positive integer $s$ such that $s^2 \leq j < (s+1)^2.$   Then $j=s^2+s+i$ for some $i\in [-s,s]$ and $z_j=i.$  It follows that for every integer $i$ there are infinitely many positive integers $j$ such that $z_j=i.$
Moreover,  $|z_j| \leq s \leq \sqrt{j}$ for all $j \geq 1.$

\smallskip

Let $f:\Z\rightarrow \mathbf N $ any function such that
$\liminf_{|n|\to \infty}f(n)\ge g$.   Let $n_0$ be the least positive integer such that $f(n)\ge g$ for all $|n|\ge n_0$.
Choose an integer $ r>1+\log_2(h^2+n_0)$.   Then
\begin{equation}\label{r}
h^2<2^{r-1}\qquad \text{ and }\qquad
n_0<2^{r-1}.
\end{equation}
Let $\gamma :\N_0 \rightarrow \N_0$ be a strictly increasing
function such that $ \gamma (0)=0$.

\smallskip

Consider the sequence $\U=\{u_i\}_{i=1}^{\infty}$ defined by
\begin{equation}\label{u}
\begin{cases}u_{2k-1}&=-m_k2^{-r},\\u_{2k}&=(h-1)m_k2^{-r}+z_k\end{cases}
\end{equation}
where $m_k=2^{2rk+\gamma (k)}$. We write
\begin{equation}  \label{U}
\U_k=\{u_{2k-1},u_{2k}\}\qquad  \text{ and
}\qquad \U_{<k}=\bigcup_{s<k}\U_{s}.
\end{equation}
Note that for all $j\le k$ we have
\begin{equation}\label{z1}
|z_j|\le \sqrt{k} < 2^k \leq 2^{\gamma (k)} < 2^{2r(k-1)+\gamma
(k)}= m_k2^{-2r}.
\end{equation}

\subsection{The recursive construction}

For any $B_h[g]$-sequence $\B$ we consider the set $T_{\gamma
}^r(\B)$ defined in \eqref{B}. Let $f:\Z \rightarrow \N$ be a
function such that $f(n) \geq g$ for $|n| \geq n_0.$ We construct
an increasing sequence $\{ \A_k \}_{k=0}^{\infty}$ of sets of
integers as follows:
\begin{equation}\label{a0}
\A_0=\{a\in T_{\gamma }^r(\B):  a \geq n_0\}
\end{equation}
 and, for $k\ge 1$,
\[
\A_k=\begin{cases} \A_{k-1}\cup \U_k &\text{ if
}r_{\A_{k-1},h}(z_k)<f(z_k)\\
\A_{k-1} &\text{ otherwise}
\end{cases}\]
where $z_k$ and $ \U_k$ are defined in \eqref{z} and \eqref{U}.

\smallskip

We shall prove that the set
\[
\A=\bigcup_{k=0}^{\infty}\A_k
\]
satisfies $r_{\A,h}(n)=f(n)$ for all integers $n$.

\begin{lemma}\label{dis}
Let $k \geq 1.$
For nonnegative integers $s$ and $t$ with $s+t \leq h,$  let
\[
\A_k^{(s,t)}=(h-s-t)\A_{k-1}+su_{2k-1}+tu_{2k}.
\]
The sets $\A_k^{(s,t)}$ are pairwise disjoint, except possibly the
sets $\A^{(0,0)}$ and $\A^{(h-1,1)}$.
\end{lemma}

\begin{proof}
If $n\in \A_k^{(s,t)}$ then
\begin{align*}
n&= a_1+\cdots +a_{h-s-t}+su_{2k-1}+tu_{2k}\\
&=a_1+\cdots +a_{h-s-t}+(t(h-1)-s)m_k2^{-r}+tz_k.
\end{align*}
If $a_i\in \A_0$, then  $\|a_i \|_{m_k}\le m_k2^{-2r}$ by
Proposition~\ref{p1}.
 If $a_i\in \U_{<k}$ then we use \eqref{u} and \eqref{z1} to obtain
 \begin{align*}
  \|a_i \|_{m_k}  \le   |a_i|    \le
(h-1)m_{k-1}2^{-r}+m_{k-1}2^{-2r}< h m_k 2^{-2r}.
\end{align*}
Therefore,
\begin{align*}
\| a_1+\cdots +a_{h-s-t}+tz_k \|_{m_k}  &\le
\|a_1 \|_{m_k}+\cdots
+ \|a_{h-s-t} \|_{m_k}+\|tz_k \|_{m_k}\\
&\le (h-s-t)m_kh2^{-2r}+tm_k2^{-2r} \\
&  \le h^2 m_k  2^{-2r}.
\end{align*}

Now suppose that $n\in \A_k^{(s',t')}$ for some $(s',t')\ne
(s,t)$. If $\{(s,t),(s',t')\} \neq \{(0,0),(h-1,1)\}$, then
 $$t(h-1)-s\ne t'(h-1)-s'$$
and
\begin{align*}
m_k2^{-r}&\le
\|\left ((t(h-1)-s)-(t'(h-1)-s')\right )m_k2^{-r}\|_{m_k}\\
&=
\|(t(h-1)-s)m_k2^{-r}-(t'(h-1)-s')m_k2^{-r}\|_{m_k}\\
&= \|\left
(n-(t(h-1)-s)m_k2^{-r}\right )-\left
(n-(t'(h-1)-s')m_k2^{-r}\right )\|_{m_k} \\
&\le \label{a3} \|
a_1+\cdots +a_{h-s-t}+tz_k\|_{m_k}+\| a_1'+\cdots
+a_{h-s'-t'}'+t'z_k\|_{m_k}\\
&\le 2 h^2 m_k 2^{-2r}.
\end{align*}
It follows that $h^2\ge 2^{r-1}$, which contradicts~(\ref{r}).  This completes the proof.
\end{proof}

\begin{lemma}\label{n0}
If $n\in \A_k^{(s,t)}$ for some $k\ge 1$ and $(s,t)\notin \{
(0,0), (h-1,1)\}$, then $|n|>n_0$.
\end{lemma}

\begin{proof}
If $n\in \A_k^{(s,t)}$, then
\[
n = a_1+\cdots +a_{h-s-t}+(t(h-1)-s)m_k2^{-r}+tz_k
\]
and
\begin{align*}
|n| &\ge  \| n \|_{m_k}\\
&=\| a_1+\cdots+a_{h-s-t}+tz_k+((h-1)t-s)m_k2^{-r}\|_{m_k}\\
&\ge \|((h-1)t-s)m_k2^{-r}\|_{m_k}-\| a_1+\cdots+a_{h-s-t}+tz_k\|_{m_k}  \\
&\ge |((h-1)t-s)m_k2^{-r}|- h^2 m_k 2^{-2r}\\
&\ge m_k2^{-r}- h^2 m_k2^{-2r}\ge m_k2^{-r-1}\ge 2^{2r}2^{-r-1}\\
&\ge 2^{r-1}>n_0,
\end{align*}
We have used that if $|((h-1)t-s)m_k2^{-r}|<m_k/2$,  then
\[
\|((h-1)t-s)m_k2^{-r}\|_{m_k}=|((h-1)t-s)m_k2^{-r}|\ge m_k2^{-r}.
\]
Also we have used $(h-1)t-s\ne 0$ and (\ref{r}) in
the last inequalities.
\end{proof}

\begin{lemma}\label{For}
For any $k\ge 0$, for any $h'<h$ and for any integer $m$ we have that
\[
r_{\A_k, h'}(m)\le g
\]
\end{lemma}

\begin{proof}
By induction on $k$.  Proposition~\ref{B[g]} implies that
$T_{\gamma}^r(\B)$ and consequently $\A_0$ are $B_h[g]$-sequences.
In particular, $\A_0$ is a $B_{h'}[g]$ sequence. Then
$r_{\A_0,h'}(m)\le g$ for any integer $m$.

 Suppose that it is true that for any $h'<h$, and for any integer
 $m$ we have that $r_{\A_{k-1},h'}(m)\le g$.

Consider $m\in h'\A_k$.
\begin{itemize}
    \item  Suppose $m\not \in
(h'-s-t)\A_{k-1}+su_{2k-1}+tu_{2k}$ for any $(s,t)\ne (0,0)$. Then
$r_{\A_k,h'}(m)= r_{\A_{k-1},h'}(m)\le g$ by induction hypothesis.
    \item Suppose that  $m\in (h'-s-t)\A_{k-1}+su_{2k-1}+tu_{2k}$ for some
$(s,t)\ne (0,0)$. Consider an element $a\in \A_0$. Then
\[
m+(h-h')a\in \A_k^{(s,t)}=(h-s-t)\A_{k-1}+su_{2k-1}+tu_{2k}.
\]
We apply lemma \ref{dis} and since $(s,t)\ne (h-1,1)$ (because
$h'<h$) we have that
\[
r_{\A_k,h'}(m)\le
r_{\A_k,h}(m+(h-h')a)=r_{\A_{k-1},h-s-t}(m+(h-h')a-su_{2k-1}-tu_{2k}),
\]
and we can apply induction hypothesis because $h-s-t<h$.
\end{itemize}

\end{proof}

\begin{prop}\label{>}
The sequence $\A$ defined above satisfies $r_{\A,h}(n) =
f(n)$ for all integers $n$.
\end{prop}

\begin{proof}
Since
\[
\underbrace{u_{2k-1}+\cdots +u_{2k-1}}_{h-1} +u_{2k}=z_k
\]
it follows that if $r_{\A_{k-1},h}(z_k)<f(z_k)$, then
$\A_k=\A_{k-1}\cup \U_k$ and
\[
r_{\A_k,h}(z_k)\ge r_{\A_{k-1},h}(z_k)+1 .
\]
For every integer $n$ there are infinitely many integers $k$ such
that $z_{k}=n$ and so $r_{\A_{k},h}(n)\ge f(n)$ for some $k$.

\smallskip

Next we show that, for every integer $k$, the sequence $\A_k$
satisfies $r_{\A_k,h}(n)\le f(n)$ for all $n$. The proof is by
induction on $k$.

\smallskip

Let $k=0$.   Since $\A_0$ is a $B_h[g]$-sequences, we have
$r_{\A_0,h}(n)\le g\le f(n)$ for $n\ge n_0$. If $n<n_0,$ then
$r_{\A_0,h}(n) = 0\le f(n).$

\smallskip

 Now, suppose that it is true for
$k-1$. In particular $r_{\A_{k-1},h}(z_k)\le f(z_k)$. If
$r_{\A_{k-1},h}(z_k)=f(z_k)$ there is nothing to prove because in
that case $\A_k=\A_{k-1}$. But if $r_{\A_{k-1},h}(z_k)\le
f(z_k)-1,$ then $\A_k=\A_{k-1}\cup \U_k=\A_{k-1}\cup \{ u_{2k-1}\}
\cup \{ u_{2k}\}$. We will assume that until the end of the proof.

 If $n\not \in h\A_k$ then $r_{\A_k,h}(n)=0\le f(n)$.

 If $n\in h\A_k$, since $\A_k=\A_{k-1}\cap \U_k$ we can write
\[
h\A_k = \bigcup_{s,t=0\atop s+t \leq h}^h \left((h-s-t)\A_{k-1} +
su_{2k-1}+tu_{2k}  \right).
\]
Then
\begin{eqnarray}\label{n}
n=a_1+\cdots +a_{h-s-t}+su_{2k-1}+tu_{2k}
\end{eqnarray}
for some $s,t$, satisfying $0\le s,t,\ s+t\le h$ and for some
$a_1,\dots ,a_{h-s-t}\in \A_{k-1}$.

\smallskip

For short we  write $r_{s,t}(n)$ for the number of solutions of
(\ref{n}).

\begin{itemize}
    \item If $n\in (h-s-t)\A_{k-1}+su_{2k-1}+tu_{2k}$ for some $(s,t)\ne
(0,0),\ (s,t)\ne (h-1,1)$ then, due to lemma \ref{dis}, we have
that $r_{\A_k,h}(n)= r_{s,t}(n). $\begin{itemize}
    \item For $0\le n\le n_0$ we have that $r_{s,t}(n)=0\le f(n)$  (due to
lemma \ref{n0}).
    \item For $n>n_0$ we apply lemma \ref{For} in the first inequality below
with $h'=h-s-t$ and $m=n-su_{2k-1}-tu_{2k}$,
\[r_{s,t}(n)=r_{\A_{k-1},h-s-t}(n-su_{2k-1}-tu_{2k})\le g\le
f(n)\]
\end{itemize}
    \item If $n\not \in (h-s-t)\A_{k-1}+su_{2k-1}+tu_{2k}$ for any $(s,t)\ne
(0,0),\ (s,t)\ne (h-1,1)$,
 then
$r_{\A_{k},h}(n)=r_{0,0}(n)+r_{h-1,1}(n).$ Notice that
$r_{0,0}(n)=r_{\A_{k-1},h}(n)$ and that $r_{h-1,1}(n)= 1$  if
$n=z_k$ and $r_{h-1,1}(n)= 0$ otherwise.
\begin{itemize}
    \item
If $n\ne z_k$, then $r_{\A_k,h}(n)=r_{A_{k-1},h}(n)\le f(n)$ by
induction hypothesis.
    \item If $n=z_k$, then $r_{\A_k,h}(n)=r_{\A_{k-1},h}(z_k)+r_{h-1,1}(z_k)\le
(f(z_k)-1)+1=f(n)$.

\end{itemize}
\end{itemize}
\end{proof}

\subsection{The density of $\A$}
Recall that $\gamma :\N_0\rightarrow \N_0$ is a strictly
increasing function with $\gamma (0)=0.$  Let $\R_{\geq 0} = \{x
\in \R : x \geq 0\}.$  We extend $\gamma $ to a strictly
increasing function $\gamma :\R_{\geq 0} \rightarrow \R_{\geq 0}.$
(For example, define $\gamma (x) = \gamma (k+1)(x-k)+\gamma
(k)(k+1-x)$ for $k \leq x \leq k+1.$)

We have
\[
\A(x)\ge \A_0(x)\ge T_{\gamma }^r(\B)(x)-n_0.
\]
Thus, to find a lower bound for $\A(x)$ it suffices to find a
lower bound for the density of $ T_{\gamma }^r(\B).$

\begin{lemma}
$ T_{\gamma }^r(\B)(x)>\B(x2^{-2r\gamma ^{-1}(\log_2x)})$.
\end{lemma}

\begin{proof}
Let $b$ be a positive integer such that
\[
b\le x2^{-2r\gamma^{-1}(\log_2x)}.
\]
Let $\ell $ be such that $2^{\gamma ({\ell})}\le b<2^{\gamma
({\ell}+1)}$. Then we can write
\begin{equation}b=\sum_{k=0}^{\ell}\sum_{i=\gamma
(k)}^{\gamma(k+1)-1}\e_i2^i.\end{equation} It follows from the
definition~\eqref{se} of the Zeros Inserting Transformation that
\begin{align*}
T_{\gamma }^r(b)
& = \sum_{k=0}^{\ell }2^{2rk}\sum_{i=\gamma (k)}^{\gamma (k+1)-1}\e_i 2^i  \\
& \leq 2^{2r{\ell}} b \\
& \leq 2^{2r\gamma^{-1}(\log_2 b)}b \\
& \leq 2^{2r (\gamma^{-1}(\log_2 b) - \gamma^{-1}(\log_2x) )}x \\
& \leq x.
\end{align*}
\end{proof}

Recall that  $\epsilon$ is a decreasing positive function defined
on $[1,\infty)$ such that $\lim_{x\rightarrow \infty} \epsilon(x)
= 0.$ We complete the proof of Theorem 1 by choosing a function
$\gamma$ that satisfies the inequality
\[
2^{-2r\gamma^{-1}(\log_2 x)} \geq \epsilon(x).
\]
It suffices to take $\gamma (x)>\log_2(\epsilon^{-1}(2^{-2rx}))$.

\end{document}